\documentclass[12pt]{amsart}
\usepackage{}

\usepackage{amsmath}
\allowdisplaybreaks
\usepackage{amsfonts}
\usepackage{amssymb}
\usepackage[all]{xy}           

\usepackage{bbding}
\usepackage{txfonts}
\usepackage{amscd}

\usepackage[shortlabels]{enumitem}
\usepackage{ifpdf}
\ifpdf
  \usepackage[colorlinks,final,backref=page,hyperindex]{hyperref}
\else
  \usepackage[colorlinks,final,backref=page,hyperindex,hypertex]{hyperref}
\fi
\usepackage{tikz}
\usepackage[active]{srcltx}

\makeatletter
\def \r {\frak r}
\def \q {\frak q}

\newtheorem{df}{Definition}[section]
\newtheorem{thm}{Theorem}[section]
\newtheorem{cor}{Corollary}[section]
\newtheorem{rem}{Remark}[section]

\newtheorem{prop}{Proposition}[section]
\newtheorem{exa}{Example}[section]
\newtheorem{lem}{Lemma}[section]

\numberwithin{equation}{section}

\begin{document}

\date{}
\title{Superalgebras  with homogeneous structures of Lie type}
\author{Sami Mabrouk and Othmen Ncib }

\address{University of Gafsa, Faculty of Sciences Gafsa, 2112 Gafsa, Tunisia}

\email{mabrouksami00@yahoo.fr, sami.mabrouk@fsgf.u-gafsa.tn}

\address{University of Gafsa, Faculty of Sciences Gafsa, 2112 Gafsa, Tunisia}

\email{othmenncib@yahoo.fr, othmen.ncib@fsgf.u-gafsa.tn}

\date{\today}

 \maketitle{}

{\bf\begin{center}{Abstract} \end{center}}In this paper, we extend the concept of Lie superalgebras to a more generalized framework called Super-Lie superalgebras. In addition, they seem to be exploring various super-generalizations of other algebraic structures, such as Super-associative, left (right) Super-Leibniz, and  Super-left(right)-symmetric superalgebras, then we give some examples and related fundamental results. The notion of Rota-Baxter operators with any parity on the Super-Lie superalgebras is given. Moreover, we study a representations of Super-Lie superalgebras and its associate dual representations. The notion of derivations of Super-Lie superalgebras is introduced thus we show that the converse of a bijective derivation defines a Rota-Baxter operator. Finally, we give a generalization of the Super-Lie superalgebras and some other structures in the ternary case which we supported this with some examples and interesting results. 
\vspace{0.5cm}

\noindent\textbf{Keywords:}  Lie superalgebra, Super-Lie superalgebra, (Super)-left-symmetric superalgebra, Rota-Baxter operator, representation,  derivation, $3$-super-Lie superalgebra.

\noindent{\textbf{Mathematics Subject Classification (2020):}}  	17B05, 	17B10, 	17B70, 17B38, 17A40.
\tableofcontents

\renewcommand{\thefootnote}{\fnsymbol{footnote}}
\footnote[0]{ Corresponding author (Sami Mabrouk): mabrouksami00@yahoo.fr, sami.mabrouk@fsgf.u-gafsa.tn}

 \section{Introduction}

Graded Lie algebras are a topic of interest in physics in the context of "supersymmetries", relating particles of differing statistics. In mathematics, graded Lie algebras have been known fore same time in the context of deformation theory.
A superalgebra is a $\mathbb Z_2$-graded vector space $\mathfrak g=\mathfrak g_{\overline{0}}\oplus\mathfrak g_{\overline{1}}$ equipped with an even bilinear map (i.e. if $x\in\mathfrak g_i,\;y\in\mathfrak g_j;\;i,j\in\mathbb Z_2=\{\overline{0},\overline{1}\}$, then $x\cdot y\in\mathfrak g_{i+j}$). A Lie superalgebra is a superalgebra $(\mathfrak g,[\cdot,\cdot])$ where
\begin{align*}
    [x,y]&=-(-1)^{ij}[y,x],\;\;\text{for}\; x\in\mathfrak g_i,\;y\in\mathfrak g_j,\\
    [x,[y,z]]&=[[x,y],z]+(-1)^{ij}[y,[x,z]],\;\;\text{for}\; x\in\mathfrak g_i,\;y\in\mathfrak g_j.
\end{align*}

In \cite{Berfzin-Katz}, the authors studies the theory of Lie superalgebras as Lie algebras of certain generalized groups,
nowadays called Lie supergroups, whose function algebras are algebras with commuting and anticommuting variables. Recently, a satisfactory theory, similar to Lie’s theory, has been developed on the connection between Lie supergroups and Lie superalgebras \cite{Berfzin-Leites}. 

Another class of non-associative superalgebras called
Lie antialgebras is introduced by V. Ovsienko in \cite{Ovsienko}.  The axioms were established after encountering two “unusual”
algebraic structures in a context of symplectic geometry: the algebras, named $\text{asl}_2$ 
and the conformal antialgebra $\mathcal{AK}(1)$. Ovsienko found that both of them are related
to an odd bivector fields on $\mathbb R^{
2|1}$
invariant under the action of the orthosymplectic
Lie superalgebra $\mathfrak{osp}(1|2)$. The algebra $\mathcal{AK}(1)$ is also related to the famous NeveuSchwartz Lie superalgebra (see \cite{LecomteOvsienko, Morier-GenoudOvsienko} for more details).

The study of Rota-
Baxter algebra appeared for the first time in the work of Baxter \cite{Baxter1} in $1960$
and were then intensively studied by Atkinson \cite{Atkinson1}, Miller \cite{Miller1}, Rota \cite{Rota1}, Cartier
\cite{Cartier1}, and more recently, they reappeared in the work of Guo \cite{Guo1} and Ebrahimi-Fard
\cite{Ebrahimi1}. A few years later, the Rota-Baxter operators were studied in the super case of some algebraic structures. We are interested in the Rota-Baxter operators of weight $\lambda\in\mathbb K$ on Lie superalgebras which is introduced in \cite{Abdaoui-Mabrouk-Makhlouf} and defined as an even  linear map $\mathfrak R:\mathfrak g\to\mathfrak g$ satisfying for any homogeneous elements $x,y\in \mathfrak g$, the identity
\begin{equation}\label{R-B-Lie-Sup}
[\mathfrak R(x),\mathfrak R(y)]=\mathfrak R\Big([\mathfrak R(x),y]+[x,\mathfrak R(y)]+\lambda[x,y]\Big).    
\end{equation}

Pre-Lie algebras (called also left-symmetric algebras, Vinberg algebras, quasi-associative
algebras) are a class of a natural algebraic systems appearing in many
fields in mathematics and mathematical physics. They were first mentioned by Cayley
in $1890$ \cite{Cayley}. They play an important
role in the study of symplectic and complex structures on Lie groups and Lie algebras
\cite{Andrada-Salamon,Chu,Dardié-Medina1,Dardié-Medina2,Lichnerowicz-Medina}, phases spaces of Lie algebras \cite{Bai1, Kupershmidt},
classical and quantum Yang–Baxter equations \cite{Diata-Medina} etc. Recently, the $\mathbb Z_2$-graded version of pre-Lie algebras called pre-Lie superalgebras are defined as a linear superspace $\mathfrak g$ equipped with an even bilinear map $"\circ"$ satisfying for any homogeneous elements $x,y\in\mathfrak g$, the identity 
\begin{equation}\label{iden-pre-L}
  \mathfrak{ass}(x,y,z)=(-1)^{|x||y|}\mathfrak{ass}(y,x,z),  
\end{equation}
where $\mathfrak{ass}(x,y,z)=x\circ (y\circ z)-(x\circ y)\circ z$, also appeared in many others fields (see \cite{Chapoton-Livernet,Gerstenhaber,Mikhalev-Mikhalev} for more details). They were first introduced by Gerstenhaber
in $1963$ to study the cohomology structure of associative algebras \cite{Gerstenhaber}.

When we look for the relationships between   Lie superalgebras and  pre-Lie superalgebras  structures we find that we can construct pre-Lie superalgebras from  Lie superalgebras uses Rota–Baxter operators. Let $\mathfrak R$ be a Rota-Baxter operator   on a Lie superalgebra $(\mathfrak g,[\cdot,\cdot])$, then, the product 
\begin{equation}\label{pre-Lie-from-Lie}
    x\circ y=[\mathfrak R(x),y],\;\forall x,y\in\mathcal{H}(\mathfrak g),
\end{equation}
defines a pre-Lie superalgebra.
The reader notices that the Rota-Baxter operator is always considered as an even linear map and he has the right to ask the following question: why is the operator $\mathfrak R$ not considered odd? In \cite{Bai-Guo-Zhang}, the authors answer this question and define the homogeneous Rota-Baxter operators   with parity as the data in the following:

 A Rota-Baxter operator on     Lie-superalgebra $(\mathfrak g,[\cdot,\cdot])$ is a linear map $\mathfrak R:\mathfrak g\to\mathfrak g$  defined by  
 
 \begin{equation}\label{G-R-B-Lie-Sup}
[\mathfrak R(x),\mathfrak R(y)]=\mathfrak R\Big((-1)^{|\mathfrak R|(|x|+|\mathfrak R|)}[\mathfrak R(x),y]+[x,\mathfrak R(y)]\Big),\;\forall x,y\in\mathcal{H}(\mathfrak g).    
\end{equation}
 
Motiveted by this works, we introduce the notion of      superalgebra with homogenous structures of Lie type. We study  this phenomenon for other algebraic structures (associative superalgebras, pre-Lie superalgebras etc) and their applications and associated operators such as representations, derivations and  Rota-Baxter operators. The ternary  extending superalgebra with homogenous structures of Lie type is studied.
\subsection{Layout of the paper} Here is an outline of the paper. In Section \ref{Sec2}, we introduce a Super-generalization of some superalgebraic structures that we are interested in, and to give some examples and related results which gives value and importance to this new structures. In Section \ref{Sec3}, we study a generalization of Rota-Baxter operators on  Super-Lie superalgebras, which requires us to provide to introduce the right and left Rota-Baxter operators of Super-Lie superalgebras which allows us to obtain a relation between Super-Lie superalgebras structures and its associated Super-left(right)-symmetric superalgebras structures via left (right) Rota-Baxter operators. We also give some interesting examples. Representation theoy of Super-Lie superalgebras and its associated dual representation are introduced in Section \ref{Sec4}, then  we  characterize it by semi-direct product. In what follows allowing us to  introduce the notions of right and left derivations of Super-Lie superalgebras and we show that, the converse
of an invertible left (resp. right) derivation define a right (resp. left) Rota-Baxter operator. We finished this section  by  a generalization of the Super-Lie superalgebras and some other structures in the ternary case which we supported this with some examples and interesting results. Section \ref{Sec5} is devoted to   give a generalization of the Super-Lie superalgebras and some other structures in the ternary case which we supported this with some examples and interesting results. The aim of Section \ref{Sec6} is to open the horizons by asking and discussing some important questions about certain notions related to these new developments, the most important of which are concerning deformations and cohomology, Maurer-Cartan characterization, $\mathcal{O}$-operators etc. 
\subsection{Conventions and notations}
\begin{enumerate}
    \item 
Throughout this paper, $\mathbb{K}$ denotes an algebraically closed field of characteristic zero, and all vector spaces are over $\mathbb{K}$ and finite-dimensional.

\item A vector space $V$ is said to be  $\mathbb{Z}_2$-graded if we are given a family $(V_i)_{i\in\mathbb{Z}_2}$ of vector subspaces of $V$ such that $V=V_{\overline{0}}\oplus V_{\overline{1}}.$
\item The symbol $\overline{x}$ always implies that $x$ is a
$\mathbb{Z}_2$-homogeneous element and it is their $\mathbb{Z}_2$-degree. We denote by $\mathcal{H}(V)$ the set of all homogeneous elements of $V$ and $\mathcal{H}(V^n)$ refers to the set of tuples with homogeneous elements. 
\item We refer by $\overline{F}$ to the parity of a multilinear map
$F : V^1\times\dots\times V^n \rightarrow W$   by
$$F(V_{i_1}^1,\dots,V_{i_n}^1)\subset W_{i_1+\cdots+i_n+\overline{F}}.$$

\item Let $End(V )$ be the $\mathbb{Z}_{2}$-graded vector space of endomorphisms of a $\mathbb{Z}_{2}$-graded vector space $V  .$ The graded binary commutator $[f,g] = f \circ g - (-1)^{\overline{f}\overline{g}}g \circ f$ induces a structure of Lie superalgebra in $End(V )$.
\end{enumerate}

\section{Super-Lie superalgebras and Admissibilty}
\label{Sec2}
In this section, we introduce the Super-generalization of some algebraic structures such as Super-associative superalgebras, left(right) Super-Leibniz superalgebras, Super-Lie superalgebras, Super-left(right)-symmetric superalgebras and we give some examples and related results of these structures. All structures studied  throughout this work are considered homogeneous.

\subsection{Definitions and Examples}

\begin{df}\label{def-graded-ass-super-alg}
A {\it\bf{Super-associative superalgebra}} is a couple $(\mathfrak{g},\mu)$ consisting of a $\mathbb{Z}_2$-graded vector space $\mathfrak{g}$ and a bilinear map $\mu:\mathfrak{g}\times\mathfrak{g}\to\mathfrak{g}$ satisfying
\begin{equation}\label{graded-super-associativity}
\mathfrak{ass}_\mu(x,y,z)=0,\;\forall x,y,z\in\mathfrak{g},\; (\mathbb{Z}_2\text{-graded associativity condition})
\end{equation}
where $\mathfrak{ass}_\mu(x,y,z)=\mu(\mu(x,y),z)-\mu(x,\mu(y,z))$. If, in addition $\mu(x,y)=(-1)^{(\overline{x}+\overline{\mu})(\overline{y}+\overline{\mu})}\mu(y,x),\;\forall x,y\in\mathcal{H}(\mathfrak{g})$, the super-associative superalgebra is called supercommutative.
\end{df}
\begin{rem}
If $\overline{\mu}=0$, we recover the classical associative superalgebras structures also called even associative superalgebras. If $\overline{\mu}=1$, we call it   odd associative superalgebras.     
\end{rem}
\begin{exa}\label{ex-n-com-assoc-sup}
Let $\mathfrak g=\mathfrak g_{\overline{0}}\oplus\mathfrak g_{\overline{1}}$ where $\mathfrak g_{\overline{0}}=<e_1>$ and $\mathfrak g_{\overline{1}}=<e_2>$ be a two dimentional $\mathbb Z_2$-graded vector space. Let use define the bilinear map $\mu:\mathfrak g\times\mathfrak g\to\mathfrak g$ on the elements of the basis of $\mathfrak g$ by
$$\mu(e_1,e_2)=-\mu(e_2,e_1)=e_1,\;\;\mu(e_2,e_2)=e_2.$$
Then $(\mathfrak g,\mu)$ is a super-noncommutative odd associative superalgebra.
\end{exa}
\begin{exa}\label{ex-com-assoc-sup}
Let $\mathfrak g=<e_1,e_2\;|\;e_3>$ be a three dimentional $\mathbb Z_2$-graded vector space. Then, the odd bilinear map $\mu:\mathfrak g\times\mathfrak g\to\mathfrak g$ defined by
$$\mu(e_1,e_2)=-\mu(e_2,e_1)= e_3,$$
  defines on $\mathfrak g$ a supercommutative odd associative superalgebra structure.
\end{exa}
    \begin{df}
A {\it\bf{left Super-Leibniz superalgebra}} is a couple $(\mathfrak{g},[\cdot,\cdot])$ where, $\mathfrak g$ is a $\mathbb{Z}_2$-graded vector space and $[\cdot,\cdot]$ be a bilinear map satisfying the following identity
\begin{equation}\label{Left-Leibniz-identity}
 [x,[y,z]]=[[x,y],z]+(-1)^{(\overline{x}+\overline{[\cdot,\cdot]})(\overline{y}+\overline{[\cdot,\cdot]})}[y,[x,z]],   
\end{equation}
for any homogeneous elements $x,y,z\in\mathfrak g$. 
The identity \eqref{Left-Leibniz-identity} is called {\it\bf{left Super-Leibniz identity}}.
\end{df}
Let us fixed an element $x\in\mathcal{H}(\mathfrak g)$. Define the linear map $ad^L_x:\mathfrak g\to\mathfrak g$ by
\begin{equation}\label{left-adj}
ad^L_x(y)=[x,y],\;\forall y\in\mathcal{H}(\mathfrak g).    
\end{equation}
Then, $\overline{ad^L_x}=\overline{x}+\overline{[\cdot,\cdot]}$ and the left Super-Leibniz identity \eqref{Left-Leibniz-identity} can be written as follow
\begin{equation}\label{Left-adj-Leibniz-identity}
ad^L_x[y,z]=[ad^L_x(y),z]+(-1)^{(\overline{x}+\overline{[\cdot,\cdot]})(\overline{y}+\overline{[\cdot,\cdot]})}[y,ad^L_x(z)].
\end{equation}
 \begin{df}
A {\it\bf{right Super-Leibniz superalgebra}} is a $\mathbb{Z}_2$-graded vector space equipped with a bilinear map $[\cdot,\cdot]$ such that, for any homogeneous elements $x,y,z\in\mathfrak g$, the following identity hold
\begin{equation}\label{right-Leibniz-identity}
 [[x,y],z]=[x,[y,z]]+(-1)^{(\overline{y}+\overline{[\cdot,\cdot]})(\overline{z}+\overline{[\cdot,\cdot]})}[[x,z],y].   
\end{equation} 
The identity \eqref{right-Leibniz-identity} is called {\it\bf{right Super-Leibniz identity}}.
\end{df}

For any fixed homogeneous element $x$ of $\mathfrak g$, we define the linear map $ad^R_x:\mathfrak g\to\mathfrak g$ by
\begin{equation}\label{right-adj}
ad^R_x(y)=[y,x],\;\forall y\in\mathcal{H}(\mathfrak g).    
\end{equation}
Then, $\overline{ad^R_x}=\overline{x}+\overline{[\cdot,\cdot]}$ and the right Super-Leibniz identity \eqref{right-Leibniz-identity} can be written as follow
\begin{equation}\label{right-adj-Leibniz-identity}
ad^R_z[x,y]=[x,ad^R_z(y)]+(-1)^{(\overline{y}+\overline{[\cdot,\cdot]})(\overline{z}+\overline{[\cdot,\cdot]})}[ad^R_z(x),y].
\end{equation}
\begin{df}\label{def-graded-Lie-super-alg}
 A  \textbf{Super-Lie superalgebra} is a $\mathbb{Z}_2$-graded vector space $\mathfrak{g}$ equipped with a bilinear product $[\cdot,\cdot]:\mathfrak{g}\times\mathfrak{g}\to\mathfrak{g}$ satisfying
 \begin{align}
& [x,y]=-(-1)^{(\overline{x}+\overline{[\cdot,\cdot]})(|y|+\overline{[\cdot,\cdot]})}[y,x]\;\;\;\;\;\;\;\;(\mathbb{Z}_2\text{-graded super-skewsymmetry}),\label{cond-graded-sup-ant}\\
& \displaystyle\circlearrowleft_{x,y,z}(-1)^{(\overline{x}+\overline{[,\cdot,\cdot]})(\overline{z}+\overline{[,\cdot,\cdot]})}[x,[y,z]]=0
\;\;(\mathbb{Z}_2\text{-graded super-Jacobi identity})\label{graded-Jac-id}
 \end{align}
 for any $x,y,z\in\mathcal{H}(\mathfrak{g})$. Such Super-Lie superalgebra is denoted by $(\mathfrak{g},[\cdot,\cdot])$.
\end{df}
\begin{rem}
If $\overline{[\cdot,\cdot]}=0$, we recover the (even) Lie superalgebra structure and if $\overline{[\cdot,\cdot]}=1$, we recall it an odd Lie superalgebra.    
A Super-Lie superalgebra is just a left Super-Leibniz (resp. right Super-Leibniz) superalgebra $(\mathfrak{g},[\cdot,\cdot])$ in which, the product $[\cdot,\cdot]$ is $\mathbb{Z}_2$-graded super-skewsymmetric.
\end{rem}

\begin{exa}
 Let $\mathfrak g=<e_1,e_2,e_3>$ be a three dimentional $\mathbb Z_2$-graded vector space where $\mathfrak g_{\overline{0}}=<e_1>$ and $\mathfrak g_{\overline{1}}=<e_2,e_3>$. Define the bilinear product  $[\cdot,\cdot]:\mathfrak g\times\mathfrak g\to\mathfrak g$ by
 $$[e_2,e_3]=e_2,$$
 and the other products are zeros. Then $(\mathfrak g,[\cdot,\cdot])$ is an odd Lie superalgebra.
\end{exa}
\subsection{Super-Lie admissible superalgebras}
\begin{df}
A \textbf{Super-Lie admissible superalgebra} is a $\mathbb{Z}_2$-graded vector space $\mathfrak g$ equipped with a bilinear map $\mu:\mathfrak g\times\mathfrak g\to\mathfrak g$ such that, the bracket defined for all homogeneous elements $x,y\in\mathfrak g$ by 
\begin{equation}\label{braket-Lie-admiss}
 [x,y]=\mu(x,y)-(-1)^{(\overline{x}+\overline{\mu})(\overline{y}+\overline{\mu})}\mu(y,x),   
\end{equation}
satisfies the $\mathbb{Z}_2$-graded super-Jacobi identity \eqref{graded-Jac-id}.
\end{df}

\begin{df}
Let $\mathfrak g$ be a  $\mathbb{Z}_2$-graded vector space. Then, $\mathfrak g$ is called:
\begin{enumerate}
\item  {\it\bf{Super-left-symmetric superalgebra}} if it is equipped with a product $\vartriangleright$ satisfying
\begin{equation}\label{left-sym-cond}
\mathfrak{ass}_\vartriangleright(x,y,z)=(-1)^{(\overline{x}+\overline{\vartriangleright})(\overline{y}+\overline{\vartriangleright})}\mathfrak{ass}_\vartriangleright(y,x,z),\;\forall x,y,z\in\mathcal{H}(\mathfrak g).    
\end{equation}
    \item  {\it\bf{Super-right-symmetric superalgebra}} if it is equipped with a product $\vartriangleleft$ satisfying
\begin{equation}\label{right-sym-cond}
\mathfrak{ass}_\vartriangleleft(x,y,z)=(-1)^{(\overline{y}+\overline{\vartriangleleft})(\overline{z}+\overline{\vartriangleleft})}\mathfrak{ass}_\vartriangleleft(x,z,y),\;\forall x,y,z\in\mathcal{H}(\mathfrak g).    
\end{equation}

\end{enumerate}
\end{df}
\begin{rem}\label{assoc-both-left-right}
    Any Super-associative superalgebra is a Super-left-symmetric (resp. Super-right-symmetric) superalgebra. 
\end{rem}
\begin{thm}\label{left-sym-to-Lie-super}
Any Super-left-symmetric (resp. Super-right-symmetric) superalgebra is a Super-Lie admissible superalgebra.
\end{thm}
\begin{proof}
Let $(\mathfrak{g},\vartriangleright)$ be a Super-left-symmetric superalgebra
Then, it is easy to see that, $[\cdot,\cdot]$ is $\mathbb{Z}_2$-graded super-skewsymmetric. It remains to show that $\mathbb{Z}_2$-graded super-Jacobi identity is satisfied.
Let $x,y,z\in\mathcal{H}(\mathfrak{g})$, we have
\begin{align*}
[x,[y,z]]&=x\vartriangleright[y,z])-(-1)^{(\overline{x}+\overline{\vartriangleright})(\overline{[y,z]}+\overline{\vartriangleright})}[y,z]\vartriangleright x)\\&=x\vartriangleright(y\vartriangleright z)-(-1)^{(\overline{y}+\overline{\vartriangleright})(\overline{z}+\overline{\vartriangleright})}x\vartriangleright(z\vartriangleright y)-(-1)^{(\overline{x}+\overline{\vartriangleright})(\overline{y}+\overline{z}+2\overline{\vartriangleright})}(y\vartriangleright z)\vartriangleright x\\&+(-1)^{(\overline{x}+\overline{\vartriangleright})(\overline{y}+\overline{z}+2\overline{\vartriangleright})}(-1)^{(\overline{y}+\overline{\vartriangleright})(\overline{z}+\overline{\vartriangleright})}(z\vartriangleright y)\vartriangleright x\\&=x\vartriangleright(y\vartriangleright z)-(-1)^{(\overline{y}+\overline{\vartriangleright})(\overline{z}+\overline{\vartriangleright})}x\vartriangleright(z \vartriangleright y)-(-1)^{(\overline{x}+\overline{\vartriangleright})(\overline{y}+\overline{z})}(y\vartriangleright z)\vartriangleright x\\&+(-1)^{(\overline{x}+\overline{\vartriangleright})(\overline{y}+\overline{z})}(-1)^{(\overline{y}+\overline{\vartriangleright})(\overline{z}+\overline{\vartriangleright})}(z\vartriangleright y)\vartriangleright x. 
\end{align*}
Similarly, we have:
\begin{align*}
 [[x,y],z]&=(x\vartriangleright y)\vartriangleright z-(-1)^{(\overline{x}+\overline{\vartriangleright})(\overline{y}+\overline{\vartriangleright})}(y\vartriangleright x)\vartriangleright z-(-1)^{(\overline{z}+\overline{\vartriangleright})(\overline{x}+\overline{z})}z\vartriangleright(x\vartriangleright y)\\&+(-1)^{(\overline{z}+\overline{\vartriangleright})(\overline{x}+\overline{y})}(-1)^{(\overline{x}+\overline{\vartriangleright})(\overline{y}+\overline{\vartriangleright})}z\vartriangleright(y\vartriangleright x)  
\end{align*}
and
\begin{align*}
 [y,[x,z]_\mu]_\mu&=y\vartriangleright(x\vartriangleright z)-(-1)^{(\overline{x}+\overline{\vartriangleright})(\overline{z}+\overline{\vartriangleright})}y\vartriangleright(z\vartriangleright x)-(-1)^{(\overline{y}+\overline{\vartriangleright})(\overline{x}+\overline{z})}(x\vartriangleright z)\vartriangleright y\\&+(-1)^{(\overline{y}+\overline{\vartriangleright})(\overline{x}+\overline{z})}(-1)^{(\overline{x}+\overline{\vartriangleright})(\overline{z}+\overline{\vartriangleright})}(z\vartriangleright x)\vartriangleright y.   
\end{align*}
By, the fact that $(\mathfrak{g},\vartriangleright)$ is $\mathbb{Z}_2$-graded left-symmetric superalgebra and  a direct computation, we get 
\begin{align*}
R&=[x,[y,z]]-[[x,y],z]-(-1)^{(\overline{x}+\overline{[\cdot,\cdot]})(\overline{y}+\overline{[\cdot,\cdot]})}[y,[x,z]]\\&=\Big(x\vartriangleright(y\vartriangleright z)-(x\vartriangleright y)\vartriangleright z\Big) +(-1)^{(\overline{x}+\overline{\vartriangleright})(\overline{y}+\overline{\vartriangleright})}\Big((y\vartriangleright x)\vartriangleright z-y\vartriangleright(x\vartriangleright z)\Big)\\&+(-1)^{(\overline{y}+\overline{\vartriangleright})(\overline{z}+\overline{\vartriangleright})}\Big((x\vartriangleright z)\vartriangleright y-x\vartriangleright(z\vartriangleright y)\Big)-(-1)^{zv+\overline{\vartriangleright})(\overline{x}+\overline{y})}\Big((z\vartriangleright x)\vartriangleright y-z\vartriangleright(x\vartriangleright y)\Big)\\&-(-1)^{(\overline{x}+\overline{\vartriangleright})(\overline{y}+\overline{z})}\Big((y\vartriangleright z)\vartriangleright x-y\vartriangleright(z\vartriangleright x)\Big)+ (-1)^{(\overline{x}+\overline{\vartriangleright})(\overline{y}+\overline{z})}(-1)^{(\overline{y}+\overline{\vartriangleright})(\overline{z}+\overline{\vartriangleright})}\Big((z\vartriangleright y)\vartriangleright x-z\vartriangleright(y\vartriangleright x)\Big)\\&=-\mathfrak{ass}_\vartriangleright(x,y,z)+(-1)^{(\overline{x}+\overline{\vartriangleright})(\overline{y}+\overline{\vartriangleright})}\mathfrak{ass}_\vartriangleright(y,x,z)+(-1)^{(\overline{y}+\overline{\vartriangleright})(\overline{z}+\overline{\vartriangleright})}\Big(\mathfrak{ass}_\vartriangleright(x,z,y)-(-1)^{(\overline{x}+\overline{\vartriangleright})(\overline{z}+\overline{\vartriangleright})}\mathfrak{ass}_\vartriangleright(z,x,y)\Big)\\&-(-1)^{(\overline{x}+\overline{\vartriangleright})(\overline{y}+\overline{z})}\Big(\mathfrak{ass}_\vartriangleright(y,z,x)-(-1)^{(\overline{y}+\overline{\vartriangleright})(\overline{z}+\overline{\vartriangleright})}\mathfrak{ass}_\vartriangleright(z,y,x)\Big)\\&=0. 
\end{align*}
Then, $(\mathfrak g,\vartriangleright)$ is a Super-Lie admissible superalgebra. Similarly, we can show the same result for the Super-right-symmetric superalgebra.   
\end{proof}

By Theorem \ref{left-sym-to-Lie-super} and Remark \ref{assoc-both-left-right}, we can get the following result.
\begin{cor}\label{sup-Lie-from-sup-ass}
Any Super-associative superalgebra is a Super-Lie admissible superalgebra.
\end{cor}
\begin{exa}
Let us consider the Super-associative superalgebra $(\mathfrak g,\mu)$ given in Example \ref{ex-n-com-assoc-sup}. Then, by Corollary \ref{sup-Lie-from-sup-ass}, the odd bilinear map $[\cdot,\cdot]_\mu$ defined on the basis $\{e_1,e_2\}$ by 
$$[e_1,e_2]\mu=\mu(e_1,e_2)-(-1)^{\overline{\mu}(1+\overline{\mu})}\mu(e_2,e_1)=2e_1,$$
defines on $\mathfrak g$ an odd Lie superalgebra structure.
\end{exa}

\begin{prop}
    Let $(\mathfrak{g},\vartriangleright)$ be a Super-left-symmetric superalgebra. Then $(\mathfrak{g},\vartriangleleft)$ is a Super-right-symmetric superalgebra, where 
    \begin{equation}
     x\vartriangleleft y=(-1)^{(\overline{x}+\overline{\vartriangleright})(\overline{y}+\overline{\vartriangleright})}y\vartriangleright x,\;\forall x,y\in\mathcal{H}(\mathfrak{g}).
 \end{equation}
\end{prop}
\begin{df}
A $\mathbb{Z}_2$-graded superalgebra $(\mathfrak g,\circ)$ is called  {\it\bf{Super-flexible superalgebra}}, if it satisfies:
\begin{equation}\label{right-sym-cond}
\mathfrak{ass}_\circ(x,y,z)=(-1)^{(\overline{x}+\overline{\circ})(\overline{y}+\overline{z})+(\overline{y}+\overline{\circ})(\overline{z}+\overline{\circ})}\mathfrak{ass}_\circ(z,y,x),\;\forall x,y,z\in\mathcal{H}(\mathfrak g).    
\end{equation}

\end{df}

\begin{prop}
If $\mathfrak g$ is both super-left and super-right-symmetric superalgebra, then it is a Super-flexible superalgebra. Moreover, any Super-flexible superalgebra is an admissible Super-Lie superalgebra.
    
\end{prop}
\begin{rem}
Any Super-associative superalgebra is a Super-flexible superalgebra.  
\end{rem}
\section{Rota-Baxter of Super-Lie superalgebras}\label{Sec3}

In this section, we give an extension of the notion of  Rota-Baxter operators on the Super-Lie Superalgebras, for this we are obliged to introduce the notions of the left Rota-Baxter operators, and the right Rota-Baxter operators on Super-Lie superalgebras wich is a generalization of what given in \cite{Bai-Guo-Zhang}. Thus, we give the relation between Super-Lie superalgebras and Super-left(right)-symmetric superalgebras from the left (right) Rota-Baxter operators   on Super-Lie superalgebras. We also enriched it with some examples.

\begin{df}\label{def-left-R-B}
Let $(\mathfrak g,[\cdot,\cdot])$ be a Super-Lie superalgebra. A \textbf{left Rota-Baxter operator}   on $\mathfrak g$ is an homogeneous linear map $\mathfrak{R}:\mathfrak g\to\mathfrak g$ satisfying
\begin{equation}\label{lef-sym-R-B}
 [\mathfrak{R}(x),\mathfrak{R}(y)]=\mathfrak{R}\Big([\mathfrak{R}(x),y]+(-1)^{\overline{\mathfrak{R}}(\overline{y}+\overline{\mathfrak{R}}+\overline{[\cdot,\cdot]})}[x,\mathfrak{R}(y)] \Big),\;\forall x,y\in\mathcal{H}(\mathfrak g).   
\end{equation}
\end{df}
\begin{thm}\label{super-left-by-sup-Lie}
Let $\mathfrak{R}$ be left Rota-Baxter operator  on a Super-Lie superalgebra $(\mathfrak g,[\cdot,\cdot])$. Then, the binary operation defined by 
\begin{equation}\label{grad-Lie-superalg-to-left-symm}
 x\vartriangleright y=[\mathfrak{R}(x),y],\;\forall x,y\in\mathcal{H}(\mathfrak g),   
\end{equation}
defines on $\mathfrak g$ a Super-left-symmetric superalgebra structure.
\end{thm}
\begin{proof}
It is easy to see that Eq. \eqref{grad-Lie-superalg-to-left-symm} gives, $\overline{\vartriangleright}=\overline{\mathfrak{R}}+\overline{[\cdot,\cdot]}$.
Let $x,y,z\in\mathcal{H}(\mathfrak g)$. Then, by using Eqs. \eqref{Left-Leibniz-identity} and \eqref{lef-sym-R-B}, we have
\begin{align*}
\mathfrak{ass}_\vartriangleright(x,y,z)-(-1)^{(\overline{x}+\overline{\vartriangleright})(\overline{y}+\overline{\vartriangleright})}\mathfrak{ass}_\vartriangleright(y,x,z)&=(x\vartriangleright y)\vartriangleright z-x\vartriangleright(y\vartriangleright z)-(-1)^{(\overline{x}+\overline{\vartriangleright})(\overline{y}+\overline{\vartriangleright})}(y\vartriangleright x)\vartriangleright z\\&+(-1)^{(\overline{x}+\overline{\vartriangleright})(\overline{y}+\overline{\vartriangleright})}y\vartriangleright(x\vartriangleright z)\\&=[\mathfrak{R}[\mathfrak{R}(x),y],z]-(-1)^{(\overline{x}+\overline{\mathfrak{R}}+\overline{[\cdot,\cdot]})(\overline{y}+\overline{\mathfrak{R}}+\overline{[\cdot,\cdot]})}[\mathfrak{R}[\mathfrak{R}(y),x],z]\\&-[\mathfrak{R}(x),[\mathfrak{R}(y),z]]+(-1)^{(\overline{x}+\overline{\mathfrak{R}}+\overline{[\cdot,\cdot]})(\overline{y}+\overline{\mathfrak{R}}+\overline{[\cdot,\cdot]})}[\mathfrak{R}(y),[\mathfrak{R}(x),z]]\\&=\Big[\Big(\mathfrak{R}[\mathfrak{R}(x),y]+(-1)^{\overline{\mathfrak{R}}(\overline{y}+\overline{\mathfrak{R}}+\overline{[\cdot,\cdot]})}\mathfrak{R}[x,\mathfrak{R}(y)]\Big),z\Big]\\&-[[\mathfrak{R}(x),\mathfrak{R}(y)],z]\\&=[[\mathfrak{R}(x),\mathfrak{R}(y)],z]-[[\mathfrak{R}(x),\mathfrak{R}(y)],z]\\&=0.    
\end{align*}
Then, $(\mathfrak g,\vartriangleright)$ is a Super-left-symmetric superalgebra.
\end{proof}
\begin{exa}\label{exp-LRB}
Let us consider the orthosymplectic superalgebra $\mathfrak{osp}(1|2)=\mathfrak{osp}(1|2)_{\overline{0}}\oplus\mathfrak{osp}(1|2)_{\overline{1}}$ by the respective basis $\{Y,F,G\}$ of $\mathfrak{osp}(1|2)_{\overline{0}}$, $\{E_+,E_-\}$ of $\mathfrak{osp}(1|2)_{\overline{1}}$ and the even bilinear product define on the basis of $\mathfrak{osp}(1|2)$ by 
$$[Y,F]=F,\;\;[Y,G]=-G,\;\;[F,G]=2Y,$$
$$[Y,E_\pm]=\pm\frac{1}{2}E_\pm,\;\;[F,E_+]=[G,E_-]=0,$$
$$[F,E_-]=-E_+,\;\;\;[G,E_+]=-E_-,$$
$$[E_+,E_+]=F,\;\;[E_-,E_-]=-G,\;\;[E_+,E_-]=Y.$$
Then $(\mathfrak{osp}(1|2),[\cdot,\cdot])$ is an even Lie superalgebra (See \cite{G-P} for more details).

Now, let us define the odd linear map $\mathfrak{R}_L:\mathfrak{osp}(1|2)\to\mathfrak{osp}(1|2)$ on the elements of basis of $\mathfrak{osp}(1|2)$ by $$\mathfrak{R}_L(F)=E_-,\;\;\;\;\mathfrak{R}_L(E_+)=-\frac{1}{2}G,$$
and the other terms are zeros. Then $\mathfrak{R}$ is a left Rota-Baxter operator   on $\mathfrak{osp}(1|2)$. By using Theorem \eqref{super-left-by-sup-Lie}, the bilinear map $\vartriangleright:\mathfrak{osp}(1|2)\times\mathfrak{osp}(1|2)\to\mathfrak{osp}(1|2)$ defined by
$$F\vartriangleright Y=[\mathfrak{R}_L(F),Y]=\frac{1}{2}E_-,\;\;F\vartriangleright F=[\mathfrak{R}_L(F),F]=E_+,$$
$$F\vartriangleright E_+=[\mathfrak{R}_L(F),E_+]=Y,\;\;F\vartriangleright E_-=[\mathfrak{R}_L(F),E_-]=-G,$$
$$E_+\vartriangleright Y=[\mathfrak{R}_L(E_+),Y]=\frac{1}{2}G,\;\;E_+\vartriangleright F=[\mathfrak{R}_L(E_+),F]=-Y,$$
$$E_+\vartriangleright E_+=[\mathfrak{R}_L(E_+),E_+]=\frac{1}{2}E_-,$$
defines an odd Super-left-symmetric superalgebra structure on $\mathfrak{osp}(1|2)$.

\end{exa}
\begin{df}\label{def-right-R-B}
Let $(\mathfrak g,[\cdot,\cdot])$ be a Super-Lie superalgebra. A \textbf{right Rota-Baxter operator}   on $\mathfrak g$ is an homogeneous linear map $\mathfrak{R}:\mathfrak g\to\mathfrak g$ satifying
\begin{equation}\label{right-sym-R-B}
 [\mathfrak{R}(x),\mathfrak{R}(y)]=\mathfrak{R}\Big((-1)^{\overline{\mathfrak{R}}(\overline{x}+\overline{\mathfrak{R}}+\overline{[\cdot,\cdot]})}[\mathfrak{R}(x),y]+[x,\mathfrak{R}(y)] \Big),\;\forall x,y\in\mathcal{H}(\mathfrak g).   
\end{equation}
\end{df}
\begin{thm}\label{super-right-by-sup-Lie}
 Let $\mathfrak{R}$ be a right Rota-Baxter operator  on a Super-Lie superalgebra $(\mathfrak g,[\cdot,\cdot])$. Then $(\mathfrak g,\vartriangleleft)$ is a Super-right-symmetric superalgebra, where the bilinear map $"\vartriangleleft"$ is defined by 
\begin{equation}\label{grad-Lie-superalg-to-right-symm}
 x\vartriangleleft y=[x,\mathfrak{R}(y)],\;\forall x,y\in\mathcal{H}(\mathfrak g).   
\end{equation}   
\end{thm}
\begin{proof}
Let $x,y,z\in\mathcal{H}(\mathfrak g)$. By using Eqs. \eqref{right-Leibniz-identity} and \eqref{right-sym-R-B}, a direct computation gives that
\begin{align*}
\mathfrak{ass}_\vartriangleleft(x,y,z)-(-1)^{(\overline{y}+\overline{\vartriangleleft})(\overline{z}+\overline{\vartriangleleft})}\mathfrak{ass}_\vartriangleleft(x,z,y)&=(x\vartriangleleft y)\vartriangleleft z-x\vartriangleleft(y\vartriangleleft z)-(-1)^{(\overline{y}+\overline{\vartriangleleft})(\overline{z}+\overline{\vartriangleleft})}(x\vartriangleleft z)\vartriangleleft y\\&+(-1)^{(\overline{y}+\overline{\vartriangleleft})(\overline{z}+\overline{\vartriangleleft})}x\vartriangleleft(z\vartriangleleft y)\\&=[[x,\mathfrak{R}(y)],\mathfrak{R}(z)]-[x,\mathfrak{R}[y,\mathfrak{R}(z)]]\\&-(-1)^{(\overline{y}+\overline{[\cdot,\cdot]}+\overline{\mathfrak{R}})(\overline{z}+\overline{[\cdot,\cdot]}+\overline{\mathfrak{R}})}[[x,\mathfrak{R}(z)],\mathfrak{R}(y)]\\&+(-1)^{(\overline{y}+\overline{[\cdot,\cdot]}+\overline{R})(\overline{z}+\overline{[\cdot,\cdot]}+\overline{\mathfrak{R}})}[x,\mathfrak{R}[z,\mathfrak{R}(y)]]\\&=-\Big([x,\mathfrak{R}[y,\mathfrak{R}(z)]]+(-1)^{\overline{\mathfrak{R}}(\overline{y}+\overline{[\cdot,\cdot]}+\overline{\mathfrak{R}})}[x,\mathfrak{R}[\mathfrak{R}(y),z]]\Big)\\&+\Big([[x,\mathfrak{R}(y)],\mathfrak{R}(z)]-(-1)^{(\overline{y}+\overline{[\cdot,\cdot]}+\overline{\mathfrak{R}})(\overline{z}+\overline{[\cdot,\cdot]}+\overline{\mathfrak{R}})}\Big)[[x,\mathfrak{R}(z)],\mathfrak{R}(y)]\\&=-[x,[\mathfrak{R}(y),\mathfrak{R}(z)]]+[[x,\mathfrak{R}(y)],\mathfrak{R}(z)]\\&+(-1)^{(\overline{y}+\overline{[\cdot,\cdot]}+\overline{\mathfrak{R}})(\overline{x}+\overline{[\cdot,\cdot]})}[\mathfrak{R}(y),[x,\mathfrak{R}(z)]]\\&= -[x,[\mathfrak{R}(y),\mathfrak{R}(z)]]+[x,[\mathfrak{R}(y),\mathfrak{R}(z)]]\\&=0.  
\end{align*}
Then, $(\mathfrak g,\vartriangleleft)$ is a Super-right-symmetric superalgebra.
\end{proof}
\begin{exa}\label{expRRB}
Let us consider the even Lie Superalgebra $\mathfrak{osp}(1|2)$ given in the Example \ref{exp-LRB}. Define the odd linear map $\mathfrak{R}_R:\mathfrak{osp}(1|2)\to\mathfrak{osp}(1|2)$ by
$$\mathfrak{R}_R(F)=E_-,\quad\mathfrak{R}_R(E_+)=\frac{1}{2}G,$$
and the other terms are zeros. Then $\mathfrak{R}_R$ is a right Rota-Baxter operator   on $\mathfrak{osp}(1|2)$. So, by using Theorem \ref{super-right-by-sup-Lie}, the bilinear map $\vartriangleleft:\mathfrak{osp}(1|2)\times\mathfrak{osp}(1|2)\to\mathfrak{osp}(1|2)$ defined by
$$Y\vartriangleleft F=[Y,\mathfrak{R}_R(F)]=-\frac{1}{2}E_-,\;\;F\vartriangleleft F=[F,\mathfrak{R}_R(F)]=-E_+,$$
$$F\vartriangleleft E_+=[F,\mathfrak{R}_R(E_+)]=Y,\;\;E_-\vartriangleleft F=[E_-,\mathfrak{R}_R(F)]=-G,$$
$$Y\vartriangleleft E_+=[Y,\mathfrak{R}_R(E_+)]=-\frac{1}{2}G,\;\;E_+\vartriangleleft F=[E_+,\mathfrak{R}_R(F)]=Y,$$
$$E_+\vartriangleleft E_+=[E_+,\mathfrak{R}_R(E_+)]=\frac{1}{2}E_-,$$
defines an odd Super-right-symmetric superalgebra structure on $\mathfrak{osp}(1|2)$.

\end{exa}
\section{Represenations and derivations of Super-Lie superlagebras}\label{Sec4}

In this section, we define the notion of representation with parity of Super-Lie superalgebras and to construct its
associated dual representation. A characterisation  of representations  on a  Super-Lie superalgebra structure via semi-direct product   has been studied. In addition, we are  required to introduce the notions of right and left derivations of Super-Lie superalgebras in which, we show that, the converse of an   invertible  left (right) derivation define a right (left) Rota-Baxter operator.  Still we have not forgotten to give some examples and associated fundamental results.
\subsection{Represenations   of Super-Lie superalgebras}
Let $(\mathfrak g, [\cdot,\cdot])$ be a Super-Lie superalgebra and $V$ be a $\mathbb Z_2$-graded vector space.
\begin{df}
 A linear map $\rho:\mathfrak g\to End(V)$ is called representation of $\mathfrak g$ if $\overline{\rho}=\overline{[\cdot,\cdot]}$ and it satisfies
\begin{equation}\label{cond-rep-grad-Lie}
    \rho([x,y])=\rho(x)\rho(y)-(-1)^{(\overline{x}+\overline{[\cdot,\cdot]})(\overline{y}+\overline{[\cdot,\cdot]})}\rho(y)\rho(x),\;\forall x,y\in\mathcal{H}(\mathfrak g).
\end{equation}
We also say that $V$ is a $\mathfrak g$-module. This representation is denoted by $(V,\rho)$.
\end{df}
\begin{exa}
For any $x\in\mathcal{H}(\mathfrak g)$, the linear map $ad:\mathfrak g\to End(\mathfrak g);\;x\mapsto ad(x)$ (or $ad_x$) defined by
$$ad_x(y)=[x,y],\;\forall y\in\mathcal{H}(\mathfrak g),$$
defines a representation on $\mathfrak g$ called adjoint representation. Indeed, for any $x,y,z\in\mathcal{H}(\mathfrak g)$, by using the identity \eqref{Left-Leibniz-identity}, we have
\begin{align*}
ad_{[x,y]}(z)&=[[x,y],z]= [x,[y,z]]-(-1)^{(\overline{x}+\overline{[\cdot,\cdot]})(\overline{y}+\overline{[\cdot,\cdot]})}[y,[x,z]]\\&=ad_x\circ ad_y(z)-(-1)^{(\overline{x}+\overline{[\cdot,\cdot]})(\overline{y}+\overline{[\cdot,\cdot]})}ad_y\circ ad_x(z).   
\end{align*}
Then, $ad[x,y]=ad_x\circ ad_y-(-1)^{(\overline{x}+\overline{[\cdot,\cdot]})(\overline{y}+\overline{[\cdot,\cdot]})}ad_y\circ  ad_x,\;\forall x,y\in\mathcal{H}(\mathfrak g)$, which gives that, the map "$ad$" is a representation of $\mathfrak g$.
\end{exa}
\begin{prop}
Let   $V$ be a $\mathbb{Z}_2$-graded vector space and $\rho:\mathfrak g\to End_{\mathbb{K}}(V)$ be a linear map. Then $(V,\rho)$ is a representation of $\mathfrak g$ if and only if, $\mathfrak{g}\ltimes_\rho V:=(\mathfrak{g}\oplus V,[\cdot,\cdot]_{\mathfrak{g}\oplus V})$ is a Super-Lie superalgebra, where the bracket
$[\cdot,\cdot]_{\mathfrak{g}\oplus V}$ is defined by
\begin{equation}\label{crochet-direct-sum}
[x+u,y+v]_{\mathfrak{g}\oplus V}=[x,y]+\rho(x)(v)-(-1)^{(\overline{x}+\overline{[\cdot,\cdot]})(\overline{y}+\overline{[\cdot,\cdot]})}\rho(y)(u),
\end{equation}
for all $x,y\in \mathcal{H}(\mathfrak{g})$ and $u,v\in\mathcal{H}(V)$. We call $\mathfrak{g}\ltimes_\rho V$ the semi-direct product of the Super-Lie superalgebra $(\mathfrak{g},[\cdot,\cdot])$ and $V$. Note that $\overline{x+u}=\overline{x}=\overline{u}$.
\end{prop}
\begin{proof}
Let $x,y\in \mathcal{H}(\mathfrak{g})$ and $u,v\in\mathcal{H}(V)$, we have
\begin{align*}
[x+u,y+v]_{\mathfrak{g}\oplus V}&=[x,y]+\rho(x)(v)-(-1)^{(\overline{x}+\overline{[\cdot,\cdot]})(\overline{y}+\overline{[\cdot,\cdot]})}\rho(y)(u)\\&=-(-1)^{(\overline{x}+\overline{[\cdot,\cdot]})(\overline{y}+\overline{[\cdot,\cdot]})}[y,x]-(-1)^{(\overline{x}+\overline{[\cdot,\cdot]})(\overline{y}+\overline{[\cdot,\cdot]})}\rho(y)(u)+\rho(x)(v)\\&=-(-1)^{(\overline{x}+\overline{[\cdot,\cdot]})(\overline{y}+\overline{[\cdot,\cdot]})}([y,x]+\rho(y)(u)-(-1)^{(\overline{x}+\overline{[\cdot,\cdot]})(\overline{y}+\overline{[\cdot,\cdot]})}\rho(x)(v))\\&= (-1)^{(\overline{x}+\overline{[\cdot,\cdot]})(\overline{y}+\overline{[\cdot,\cdot]})} [y+v,x+u]_{\mathfrak{g}\oplus V}.  
\end{align*}
Then, the product $[\cdot,\cdot]_{\mathfrak{g}\oplus V}$ is $\mathbb{Z}_2$-graded super-antisymmetric. 
It remains to show that, the identity \eqref{Left-Leibniz-identity} is satisfied. Let $x,y,z\in \mathcal{H}(\mathfrak{g})$ and $u,v,w\in\mathcal{H}(V)$, then by a direct computation, we have:
\begin{align*}
[x+u,[y+v,z+w]_{\mathfrak{g}\oplus V}]_{\mathfrak{g}\oplus V}&=[x,[y,z]]+\rho(x)\Big(\rho(y)w-(-1)^{(\overline{y}+\overline{[\cdot,\cdot]})(\overline{z}+\overline{[\cdot,\cdot]})}\rho(z)v\Big)\\&-(-1)^{(\overline{x}+\overline{[\cdot,\cdot]})(\overline{y}+\overline{z})}\rho([y,z])u,    
\end{align*}

\begin{align*}
[[x+u,y+v]_{\mathfrak{g}\oplus V},z+w]_{\mathfrak{g}\oplus V}&=[[x,y],z]+\rho([x,y])w\\&-(-1)^{(\overline{x}+\overline{y})(\overline{z}+\overline{[\cdot,\cdot]})}\rho(z)\Big(\rho(x)v-(-1)^{(\overline{x}+\overline{[\cdot,\cdot]})(\overline{y}+\overline{[\cdot,\cdot]})}\rho(y)u\Big),
\end{align*}
and
\begin{align*}
 [y+v,[x+u,z+w]_{\mathfrak{g}\oplus V}]_{\mathfrak{g}\oplus V}&=[y,[x,z]]+\rho(y)\Big(\rho(x)w-(-1)^{(\overline{x}+\overline{[\cdot,\cdot]})(\overline{z}+\overline{[\cdot,\cdot]})}\rho(z)u\Big)\\&-(-1)^{(\overline{y}+\overline{[\cdot,\cdot]})(\overline{x}+\overline{z})}\rho([x,z])v.   
\end{align*}
Then, $\mathfrak{g}\ltimes_\rho V:=(\mathfrak{g}\oplus V,[\cdot,\cdot]_{\mathfrak{g}\oplus V})$ is a Super-Lie superalgebra if and only if

\begin{align*}
[x+u,[y+v,z+w]_{\mathfrak{g}\oplus V}]_{\mathfrak{g}\oplus V}&=[[x+u,y+v]_{\mathfrak{g}\oplus V},z+w]_{\mathfrak{g}\oplus V}\\&+(-1)^{(\overline{x}+\overline{[\cdot,\cdot]})(\overline{y}+\overline{[\cdot,\cdot]})}[y+v,[x+u,z+w]_{\mathfrak{g}\oplus V}]_{\mathfrak{g}\oplus V},
\end{align*}
for any, $x,y,z\in\mathcal{H}(\mathfrak g)$ and $u,v,w\in\mathcal{H}(V)$. Then, by applying the fact that $(\mathfrak g,[\cdot,\cdot])$ is a Super-Lie superalgebra and we take $u=v=0$, we get
$$\rho[x,y]w=\rho(x)\rho(y)w-(-1)^{(\overline{x}+\overline{[\cdot,\cdot]})(\overline{y}+\overline{[\cdot,\cdot]})}\rho(y)\rho(x)w$$
for any homogeneous element $w\in V$, which gives the result.
\end{proof}
Let   $(V,\rho)$ be a representation of $\mathfrak g$. Define $\rho^*:\mathfrak g \to End(V^*)$ by
\begin{equation}\label{DulaRep}
    \rho^*(x)(\xi)=-(-1)^{\overline{\xi}(\overline{x}+\overline{\rho})}\xi\circ\rho(x), \quad \forall\quad x\in\mathcal{H}(\mathfrak g)\quad \text{and }\quad \xi\in \mathcal{H}(V^*).
\end{equation}

\begin{prop}

    With the above notations, $\rho^*$ is a represenation of $\mathfrak g$ on the dual space $V^*$, which is called the dual representation of $\rho$.
\end{prop}
\begin{proof}
 Let $x,y\in\mathcal{H}(\mathfrak g)$ and $\xi\in\mathcal{H}(V^*)$. Then, by using Eq. \eqref{cond-rep-grad-Lie}, we have 
 \begin{align*}
\rho^*([x,y])(\xi)&=-(-1)^{\overline{\xi}(\overline{x}+\overline{y}+\overline{\rho}+\overline{[\cdot,\cdot]})}\xi\rho[x,y]\\&=-(-1)^{\overline{\xi}(\overline{x}+\overline{y})}\xi\big(\rho(x)\rho(y)-(-1)^{(\overline{x}+\overline{\rho})(\overline{y}+\overline{\rho})}\rho(y)\rho(x)\big)\\&=-(-1)^{(\overline{x}+\overline{\rho})(\overline{y}+\overline{\rho})}\rho^*(y)\rho^*(x)\xi+\rho^*(x)\rho^*(y)\xi.    
 \end{align*}
 Then, $\rho^*([x,y])=\rho^*(x)\rho^*(y)-(-1)^{(\overline{x}+\overline{\rho})(\overline{y}+\overline{\rho})}\rho^*(y)\rho^*(x),\;\forall x,y\in\mathcal{H}(\mathfrak g)$. Thus, $\rho^*$ is a representation of $\mathfrak g$ on $V^*$. The proof is finished.
\end{proof}
\subsection{Derivations of Super-Lie superalgebras}
\begin{df}
A left derivation of a Super-Lie superalgebra $(\mathfrak g,[\cdot,\cdot])$ is an homogeneous linear map $\mathfrak D^L:\mathfrak g\to\mathfrak g$ satisfying the following relation
\begin{equation}\label{def-left-deriv-SLie-sup}
\mathfrak D^L[x,y]=[\mathfrak D^L(x),y]+(-1)^{\overline{\mathfrak D^L}(\overline{x}+\overline{[\cdot,\cdot]})}[x,\mathfrak D^L(y)]    
\end{equation}
for all $x,y\in\mathcal{H}(\mathfrak g)$. We denoted by $\mathfrak Der^L(\mathfrak g)$, the set of all left derivations of $\mathfrak g$.
\end{df}
\begin{lem}
The $\mathbb{Z}_2$-graded vector space  $\mathfrak Der^L(\mathfrak g)$ with commutator define a Lie superalgebra.
\end{lem}
\begin{proof}
Let $\mathfrak D^L_1,\;\mathfrak D^{L}_2\in\mathfrak Der^L(\mathfrak g)$ and $\mathfrak D=[\mathfrak D^L_1 ,\mathfrak D^{L}_2]$. Then, for any  $x,y\in\mathcal{H}(\mathfrak g)$, we have 
\begin{align*}
    \mathfrak D[x,y]&=[\mathfrak D^L_1 ,\mathfrak D^{L}_2][x,y]\\&=\mathfrak D^L_1 (\mathfrak D^{L}_2[x,y])-(-1)^{\overline{\mathfrak D^L_1}\overline{\mathfrak D^L_2}}\mathfrak D^L_2 (\mathfrak D^{L}_1[x,y])\\&=\mathfrak D^L_1([\mathfrak D_2^L(x),y]+(-1)^{\overline{\mathfrak D_2^L}(\overline{x}+\overline{[\cdot,\cdot]})}[x,\mathfrak D_2^L(y)])\\&-(-1)^{\overline{\mathfrak D^L_1}\overline{\mathfrak D^L_2}}\mathfrak D^L_2([\mathfrak D_1^L(x),y]+(-1)^{\overline{\mathfrak D_1^L}(\overline{x}+\overline{[\cdot,\cdot]})}[x,\mathfrak D_1^L(y)])\\&=[\mathfrak D_1^L\mathfrak D^L_2(x),y]+(-1)^{\overline{\mathfrak D_1^L}(\overline{x}+\overline{\mathfrak D_2^L}+\overline{[\cdot,\cdot]})}[\mathfrak D_2^L(x),\mathfrak D_1^L(y)]+(-1)^{\overline{\mathfrak D_2^L}(\overline{x}+\overline{[\cdot,\cdot]})}[\mathfrak D_1^L(x),\mathfrak D_2^L(y)]\\&+(-1)^{(\overline{\mathfrak D_1^L}+\overline{\mathfrak D_2^L})(\overline{x}+\overline{[\cdot,\cdot]})}[x,\mathfrak D_1^L\mathfrak D_2^L(y)]-(-1)^{\overline{\mathfrak D^L_1}\overline{\mathfrak D^L_2}}[\mathfrak D^L_2\mathfrak D_1^L(x),y]-(-1)^{\overline{\mathfrak D^L_2}(\overline{x}+\overline{[\cdot,\cdot]})}[\mathfrak D_1^L(x),\mathfrak D^L_2(y)]\\&-(-1)^{\overline{\mathfrak D_1^L}(\overline{x}+\overline{\mathfrak D_2^L}+\overline{[\cdot,\cdot]})}[\mathfrak D_2^L(x),\mathfrak D_1^L(y)]-(-1)(-1)^{(\overline{\mathfrak D_1^L}+\overline{\mathfrak D_2^L})(\overline{x}+\overline{[\cdot,\cdot]})}[x,\mathfrak D_2^L\mathfrak D_1^L(y)]\\&=[(\mathfrak D_1^L\mathfrak D_2^L-(-1)^{\overline{\mathfrak D_1^L}\overline{\mathfrak D_2^L}}\mathfrak D_2^L\mathfrak D_1^L)(x),y]+(-1)^{(\overline{\mathfrak D_1^L}+\overline{\mathfrak D_2^L})(\overline{x}+\overline{[\cdot,\cdot]})}[x,(\mathfrak D_1^L\mathfrak D_2^L-(-1)^{\overline{\mathfrak D_1^L}\overline{\mathfrak D_2^L}}\mathfrak D_2^L\mathfrak D_1^L)(y)]\\&=[\mathfrak D(x),y]+(-1)^{\overline{\mathfrak D}(\overline{x}+\overline{[\cdot,\cdot]})}[x,\mathfrak D(y)].
\end{align*}
Then, $\mathfrak D=[\mathfrak D_1^L,\mathfrak D_2^L]\in\mathfrak Der^L(\mathfrak g)$, which gives that, $\mathfrak Der^L(\mathfrak g)$ is a Lie sub-superalgebra of $End(\mathfrak g)$. Therefore, $\mathfrak Der^L(\mathfrak g)$ is a Lie superalgebra. 
\end{proof}

\begin{exa}
For any fixed homogeneous element $x$ of a Super-Lie superalgebra $(\mathfrak g,[\cdot,\cdot])$, \\ the map $ad^L_x$ defined by Eq. \eqref{left-adj} define a left derivation of $\mathfrak g$, with the consideration that\\ $\overline{ad_x^L}=\overline{x}+\overline{[\cdot,\cdot]}$ as an immediate relation by the definition of $ad^L_x$.     
\end{exa}

\begin{df}
A right derivation of a Super-Lie superalgebra $(\mathfrak g,[\cdot,\cdot])$ is an homogeneous linear map $\mathfrak D^R:\mathfrak g\to\mathfrak g$ such that, the following condition hold
\begin{equation}\label{def-left-deriv-SLie-sup}
\mathfrak D^R[x,y]=[x,\mathfrak D^R(y)]+(-1)^{\overline{\mathfrak D^R}(\overline{y}+\overline{[\cdot,\cdot]})}[\mathfrak D^R(x),y]    
\end{equation}
for all $x,y\in\mathcal{H}(\mathfrak g)$. The set of all right derivations of $g$ is generally denoted by $\mathfrak Der^R(\mathfrak g)$.
\end{df}
\begin{exa}
For any fixed homogeneous element $x$ of a Super-Lie superalgebra $(\mathfrak g,[\cdot,\cdot])$, the map $ad^R_x$ defined by Eq. \eqref{right-adj} define a right derivation of $\mathfrak g$.     
\end{exa}
\begin{lem}
Let $(\mathfrak g,[\cdot,\cdot])$ be a Super-Lie superalgebra, then $(\mathfrak Der^R(\mathfrak g),[\cdot,\cdot])$ is a  Lie superalgebra.
\end{lem}

\begin{thm}
 Let $\mathfrak D$ be an homogeneous invertible linear map of a Super-Lie superalgebra $(\mathfrak g,[\cdot,\cdot])$. Then: 
\begin{enumerate}
\item
$\mathfrak D$ is a left drivation of $\mathfrak g$ if and only if $\mathfrak D^{-1}$ is a right Rota-Baxter operator   on $\mathfrak g.$
\item $\mathfrak D$ is a right drivation of $\mathfrak g$ if and only if $\mathfrak D^{-1}$ is a left Rota-Baxter operator   on $\mathfrak g.$
\end{enumerate}    
\end{thm}
\begin{proof}
\begin{enumerate}
\item Let $\mathfrak D$ be a left derivation of a Super-Lie superalgebra $(\mathfrak g,[\cdot,\cdot])$, then, for any two homogeneous elements $x,y\in\mathfrak g$, we have
$$\mathfrak D[x,y]=[\mathfrak D(x),y]+(-1)^{\overline{\mathfrak D}(\overline{x}+\overline{[\cdot,\cdot]})}[x,\mathfrak D(y)].$$    
Then $\mathfrak D$ is invertible if and only if
$$[x,y]=\mathfrak D^{-1}[\mathfrak D(x),y]+(-1)^{\overline{\mathfrak D}(\overline{x}+\overline{[\cdot,\cdot]})}\mathfrak D^{-1}[x,\mathfrak D(y)].$$
Suppose that $X=\mathfrak D(x)$ and $Y=\mathfrak D(y)$, we obtain
$$[\mathfrak D^{-1}(X),\mathfrak D^{-1}(Y)]=\mathfrak D^{-1}([X,\mathfrak D^{-1}(Y)]+(-1)^{\overline{\mathfrak D^{-1}}(\overline{\mathfrak D^{-1}}+\overline{X}+\overline{[\cdot,\cdot]})}[\mathfrak D^{-1}(X),Y]),$$
since the fact that $\overline{\mathfrak D}=\overline{\mathfrak D^{-1}}$. Then $\mathfrak D^{-1}$ is a right Rota-Baxter operator   on $\mathfrak g$.
\item This is a direct computation by a similar way given in $(1)$.
\end{enumerate}    
\end{proof}

\section{$3$-super-Lie superalgebras }
\label{Sec5}
In this section we introduce the case of $\mathbb Z_2$-graded $3$-Lie superalgebras  which called $3$-Super-Lie superalgebras in which we give some properties, examples and related interesting results. 
 \begin{df}
A {\it\bf{left ternary Super-Leibniz superalgebra}} is a couple $(\mathfrak{g},[\cdot,\cdot,\cdot])$ where, $\mathfrak g$ is a $\mathbb{Z}_2$-graded vector space and $[\cdot,\cdot,\cdot]$ be an homogeneous trilinear map with degree $\q\in\mathbb Z_2$ satisfying for any homogeneous elements $x,y,z\in\mathfrak g$ the following identity
\begin{align}\label{Left-3-Leibniz-identity}
 [x_1,x_2,[y_1,y_2,y_3]]&=[[x_1,x_2,y_1],y_2,y_3]+(-1)^{(\overline{x_1}+\overline{x_2}+\q)(\overline{y_1}+\q)}[y_1,[x_1,x_2,y_2],y_3]\\&+(-1)^{(\overline{x_1}+\overline{x_2}+\q)(\overline{y_1}+\overline{y_2})}[y_1,y_2,[x_1,x_2,y_3]]\nonumber.  
\end{align}
 
The identity \eqref{Left-3-Leibniz-identity} is called {\it\bf{left ternary Super-Leibniz identity}}.
\end{df}
Let $x_1,x_2$ be two fixed homogeneous elements of $\mathfrak g$. Define the linear map $ad^L_{x_1,x_2}:\mathfrak g\to\mathfrak g$ by
\begin{equation}\label{left-ter-adj}
ad^L_{x_1,x_2}(y)=[x_1,x_2,y],\;\forall y\in\mathcal{H}(\mathfrak g).    
\end{equation}
Then, $\overline{ad^L_{x_1,x_2}}=\overline{x_1}+\overline{x_2}+\q$ and the left ternary Super-Leibniz identity \eqref{Left-3-Leibniz-identity} can be written as follow
\begin{align}\label{Left-adj-Leibniz-identity}
ad^L_{x_1,x_2}[y_1,y_2,y_3]&=[ad^L_{x_1,x_2}(y_1),y_2,y_3]+(-1)^{(\overline{x_1}+\overline{x_2}+\q)(\overline{y_1}+\q)}[y_1,ad^L_{x_1,x_2}(y_2),y_3]\\&+(-1)^{(\overline{x_1}+\overline{x_2}+\q)(\overline{y_1}+\overline{y_2})}[y_1,y_2,ad^L_{x_1,x_2}(y_3)].\nonumber
\end{align}
 \begin{df}
A {\it\bf{right ternary Super-Leibniz superalgebra}} is a couple $(\mathfrak{g},[\cdot,\cdot,\cdot])$ consisting of a $\mathbb{Z}_2$-graded vector space $\mathfrak g$  and an homogeneous trilinear map $[\cdot,\cdot,\cdot]$ on $\mathfrak g$ with degree $\q\in\mathbb Z_2$ satisfying, for any homogeneous elements $x,y,z\in\mathfrak g$, the following identity
\begin{align}\label{right-3-Leibniz-identity}
 [[x_1,x_2,y_1],y_2,y_3]&=[x_1,x_2,[y_1,y_2,y_3]]+(-1)^{(\overline{y_2}+\overline{y_3}+\q)(\overline{y_1}+\q)}[x_1,[x_2,y_2,y_3],y_1]\\&+(-1)^{(\overline{x_2}+\overline{y_1})(\overline{y_2}+\overline{y_3}+\q)}[[x_1,y_2,y_3],x_2,y_1]\nonumber.  
\end{align}
 
The identity \eqref{right-3-Leibniz-identity} is called {\it\bf{right ternary Super-Leibniz identity}}.
\end{df}
Let $x_1,x_2$ be two fixed homogeneous elements of $\mathfrak g$. Define the linear map $ad^R_{x_1,x_2}:\mathfrak g\to\mathfrak g$ by
\begin{equation}\label{left-adj}
ad^R_{x_1,x_2}(y)=[y,x_1,x_2],\;\forall y\in\mathcal{H}(\mathfrak g).    
\end{equation}
Then, $\overline{ad^R_{x_1,x_2}}=\overline{x_1}+\overline{x_2}+\q$ and the left ternary Super-Leibniz identity \eqref{Left-3-Leibniz-identity} can be written as follow
\begin{align}\label{right-adj-ter-Leibniz-identity}
ad^R_{y_2,y_3}[x_1,x_2,y_1]&=[x_1,x_2,ad^R_{y_2,y_3}(y_1)]+(-1)^{(\overline{y_2}+\overline{y_3}+\q)(\overline{y_1}+\q)}[x_1,ad^R_{y_2,y_3}(x_2),y_1]\\&+(-1)^{(\overline{x_2}+\overline{y_1})(\overline{y_2}+\overline{y_3}+\q)}[ad^R_{y_2,y_3}(x_1),x_2,y_1].\nonumber
\end{align}

\begin{df}\label{def-3-sup-Lie}
Let $\mathfrak g=\mathfrak g_{\overline{0}}\oplus\mathfrak g_{\overline{1}}$ be a $\mathbb{Z}_2$-graded vector space. A trilinear mapping $(x,y,z)\in\mathfrak g\otimes\mathfrak g\otimes\mathfrak g\mapsto [x,y,z]\in\mathfrak g$ is said to be an homogeneous ternary Super-Lie bracket with degree $\q\in\mathbb Z_2$ if $\overline{[x,y,z]}=\overline{x}+\overline{y}+\overline{z}+\q$, it is $\mathbb{Z}_2$-graded super-skewsymmetric, i.e.,
\begin{equation}\label{ter-super-skew-sym}
    [x,y,z]=-(-1)^{(\overline{x}+\q)(\overline{y}+\q)}[y,x,z]=-(-1)^{(\overline{y}+\q)(\overline{z}+\q)}[x,z,y],
\end{equation}
and satisfies the $\mathbb{Z}_2$-graded Filippov-Jacobi identity
\begin{align}
[x,y,[z,t,u]]&=[[x,y,z],t,u]+(-1)^{(\overline{x}+\overline{y}+\q)(\overline{z}+\q)}[z,[x,y,t],u]\label{graded-Filippov-ident}\\&+(-1)^{(\overline{x}+\overline{y}+\q)(\overline{z}+\overline{t})}[z,t,[x,y,u]] \nonumber   
\end{align}

A vector superspace $\mathfrak g$, together with an homogeneous ternary Super-Lie bracket defined on it, is called a
$3$-Super-Lie superalgebra.
\end{df}
\begin{rem}
  If $\q=0$, we recover the (even) $3$-Lie superalgebra structure (see \cite{CantKac} for more details) and if $\q=1$, which is called  odd $3$-Lie superalgebra.  
\end{rem}
\begin{exa}\label{ex-odd-3-Lie}
 Let $\mathfrak g=<e_0,e_1>$ be a two dimentional $\mathbb Z_2$-graded vector space where $\mathfrak g_{\overline{0}}=<e_0>$ and $\mathfrak g_{\overline{1}}=<e_1>$. Define the trilinear product  $[\cdot,\cdot,\cdot]:\mathfrak g\times\mathfrak g\to\mathfrak g$ by
 $$[e_1,e_1,e_1]=e_0,$$
 and the other products are zeros. Then $(\mathfrak g,[\cdot,\cdot,\cdot])$ defines an odd $3$-Lie superalgebra.
\end{exa}

\begin{df}
Let $(\mathfrak g,[\cdot,\cdot,\cdot])$ be a $3$-Super-Lie superalgebra. An homogeneous linear map $\mathfrak{R}:\mathfrak g\to\mathfrak g$ is called \textbf{ Rota-Baxter operator} (called also \textbf{left Rota-Baxter operator}) with degree $\r\in\mathbb Z_3$ on $\mathfrak g$, if it satisfies the following rule 
\begin{align}\label{l-R-B-ternary}
[\mathfrak{R}x,\mathfrak{R}y,\mathfrak{R}z]&=\mathfrak{R}\Big([\mathfrak{R}x,\mathfrak{R}y,z]+(-1)^{\r(\r+\overline{y}+\overline{[\cdot,\cdot,\cdot]})}[\mathfrak{R}x,y,\mathfrak{R}z]+(-1)^{\r(\overline{x}+\overline{y})}[x,\mathfrak{R}y,\mathfrak{R}z]\Big)
\end{align}
for any homogeneous elements $x,y,z$ of $\mathfrak g$.
\end{df}
By a straightforward computation, we obtain the following results 
\begin{prop}
Let $\mathfrak{R}$ be a Rota-Baxter operator on  a $3$-Super-Lie superalgebra $(\mathfrak g,[\cdot,\cdot,\cdot])$. Then $(\mathfrak g,[\cdot,\cdot,\cdot]_{\mathfrak{R}})$ is a $3$-Super-Lie superalgebra, where  
\begin{align*}\label{l-R-B-ternary}
[ x, y,z]_{\mathfrak{R}}&= [\mathfrak{R}x,\mathfrak{R}y,z]+(-1)^{\r(\r+\overline{y}+\overline{[\cdot,\cdot,\cdot]})}[\mathfrak{R}x,y,\mathfrak{R}z]+(-1)^{\r(\overline{x}+\overline{y})}[x,\mathfrak{R}y,\mathfrak{R}z]
\end{align*}
for any homogeneous elements $x,y,z$ of $\mathfrak g$. Moreover, $\overline{[\cdot,\cdot,\cdot]_{\mathfrak{R}}}=\overline{[\cdot,\cdot,\cdot]}$ and ${\mathfrak{R}}$ is $3$-super-Lie superalgeras morphism from $(\mathfrak g,[\cdot,\cdot,\cdot]_{\mathfrak{R}})$ to $(\mathfrak g,[\cdot,\cdot,\cdot ])$.
\end{prop}
\begin{thm}\label{3-Lie-to-Lie}
Let $(\mathfrak g,[\cdot,\cdot,\cdot])$ be a $3$-Super-Lie superalgebra and $a\in\mathcal{H}(\mathfrak g)$ such that $\overline{a}=\overline{[\cdot,\cdot,\cdot]}=\q\in\mathbb{Z}_2$. Then $\mathfrak g_a=(\mathfrak g,[\cdot,\cdot]_a=[a,\cdot,\cdot])$ is a Lie superalgebra. Morover $a\in \mathcal Z({\mathfrak {g_a}})=\{x\in\mathfrak g;\;[x,y]_a=0,\;\forall y\in\mathfrak g \}$.
\end{thm}
\begin{proof}
It's obvious to show that $[\cdot,\cdot]_a$ is $\mathbb{Z}_2$-graded super-skewsymmetric and $[a,x,a]=[a,a,x]=0,\;\forall x\in\mathfrak g$. Let $x,y,z\in\mathcal{H}(\mathfrak g)$, then, by using \eqref{graded-Filippov-ident}, we have 
\begin{align*}
[x,[y,z]_a]_a=&[a,x,[a,y,z]]\\=&[[a,x,a],y,z]+(-1)^{(\overline{a}+\overline{[\cdot,\cdot,\cdot]})(\overline{a}+\overline{x}+\overline{[\cdot,\cdot,\cdot]})}[a,[a,x,y],z]\\&+(-1)^{(\overline{a}+\overline{x}+\overline{[\cdot,\cdot,\cdot]})(\overline{a}+\overline{y}+\overline{[\cdot,\cdot,\cdot]})}[a,y,[a,x,z]]\\=&[[x,y]_a,z]_a+(-1)^{\overline{x}\overline{y}}[y,[x,z]].    
\end{align*}
This completed the proof.
\end{proof}
\begin{exa}
Let us consider the odd $3$-Lie superalgebra $\mathfrak g=<e_0,e_1>$ given in the Example \ref{ex-odd-3-Lie}. Then, by Theorem \ref{3-Lie-to-Lie}, the bilinear map $[\cdot,\cdot]_{e_1}$ given by $$[e_1,e_1]_{e_1}=e_0,$$
defines on $\mathfrak g$ a Lie superalgebra structure.
\end{exa}

\section{Further discussion}\label{Sec6}

In this paper, we introduce a generalization of Lie superalgebras, when the bracket is not necessary even.
We study some related algebraic structure called Super-Lie admissible superalgebras constructed by Rota-Baxter operators. Their representations are studied and characterized  by semi-direct product on the sum direct of superalgebra and super-module. Here we collect some further questions regarding this new algebraic structure:
\vspace{0.5cm}

{\bf Deformations and Cohomolgy}: In Section \ref{Sec4}, we study representations of the new algebraic structures called Super-Lie superalgebras as a generalization  of Lie superalgebras. In \cite{Amor&Pinczon} the authors study deformations and cohomlogy of  Lie superalgebra with a given representation. Motivated by this work, will this allow us to generalize the studies of cohomology and deformations to the Super-Lie superalgebras?  
\vspace{0.5cm}

{\bf Maurer-Cartan characterization}: In \cite{Amor&Pinczon}, the authors  give Maurer-Cartan characterizations as well as a cohomology theory
of  Lie superalgebras. Explicitly, they  introduce the notion of a bidifferential graded
Lie algebra and thus give Maurer-Cartan characterizations of   Lie superalgebras. One may now ask the following question: Can the graded Lie algebra introduced by the authors characterize the cohomolgy of Super-Lie superalgebra when the bracket is odd?
\vspace{0.5cm}

{\bf $\mathcal O$-operators}: An $\mathcal O$-operator is a relative generalization of   Rota-Baxter operators on a Lie superalgebra associated to a given representation (see \cite{Bai-Guo-Zhang}).   A homogeneous linear map $T: V \longrightarrow
{\mathfrak  g}$ is called an  $\mathcal {O}$-operator of a Lie superalgebra
$({\mathfrak  g},[\cdot,\cdot])$ associated to  the representation $(V,\rho)$ if it satisfies
\begin{equation} \label{eq:oop}
[T(v), T(w)] = T\Big((-1)^{(|T|+|v|)|T|}\rho(T(v))w-
(-1)^{|v|(|T|+|w|)}\rho(T(w))v\Big),\;\;\forall v, w \in \mathcal{H}(V).
\end{equation}
 In Section \ref{Sec3}, we introduce a generalization of Rota-Baxter operator on a Lie superalgebra to  Super-Lie superalgebra case. It is interesting to generalize this notion for the relative case and study  their cohomolgies and deformations.

\vspace{0.5cm}

{\bf Classification of odd Rota-Baxter operators}: In \cite{Abdaoui-Mabrouk-Makhlouf, AmorAthmouniBen HassineChtiouiMabrouk}, the authors  give a classification of the  even homogeneous Rota-Baxter operators   on a Lie  superalgebras and  construct pre-Lie superalgebras from the homogeneous Rota-Baxter operators  and their subadjacent Lie superalgebras. We can also classify the  odd homogeneous Rota-Baxter operators   on  Lie  superalgebras like Witt superalgebras and Witt superalgebras of Block type. We   construct  their associated subadjacent odd Lie superalgebras.

\vspace{0.5cm}

{\bf Induced $3$-super-Lie superalgebras}: In \cite{ArMakhSilv}, the structure   of $3$-Lie algebras induced by Lie algebras have been
investigated. The super case of this   construction  using an even trace function is studied by V. Abramov in \cite{Abramov} (see more \cite{InducedBihomSuper, GenDer}).  By this works,  can we induce an odd $3$-Lie superalgebra using an odd trace function on a Lie superalgebra?

\section*{Acknowledgment}

The   authors would like to thank Prof. \textbf{Abdenacer Makhlouf} and Prof. \textbf{Taoufik Chtioui}  for helpful discussions and suggestions.

\end{document}